\documentclass{amsart}
\usepackage{amscd, amsmath, amsthm, amssymb}
\usepackage{enumerate, comment}
\usepackage{tikz}
\usepackage{pgfplots}
\usepackage[utf8]{inputenc}
\usepgfplotslibrary{ternary}
\usepackage{ifthen} 
\usepackage[bookmarksnumbered,colorlinks, plainpages]{hyperref}
\usepackage[capitalize,noabbrev]{cleveref} 
\usepackage{stackrel}
\hypersetup{
pdftitle={Research Paper},
pdfauthor={Reza Abdolmaleki},
pdfsubject={The non-pure dual exchange property of monomial ideals},
pdfcreator={Reza Abdolmaleki},
colorlinks,
linkcolor=blue,
citecolor=magenta,
anchorcolor=red,
bookmarksopen,
urlcolor=black,
filecolor=red,
}
\theoremstyle{plain}
\newtheorem{Theorem}{Theorem}[section]

\newtheorem{thm}[Theorem]{Theorem}

\newtheorem{cor}[Theorem]{Corollary}
\newtheorem{prop}[Theorem]{Proposition}

\theoremstyle{definition}
\newtheorem{Definition}[Theorem]{Definition}

\newtheorem{defn}[Theorem]{Definition}

\newtheorem{ex}[Theorem]{Example}
\newtheorem{rem}[Theorem]{Remark}
\newtheorem{setup}[Theorem]{Setup}

\theoremstyle{remark}
\newtheorem*{Acknowledgments}{Acknowledgments}

\newcommand{\KK}{\mathbb{K}}

\newcommand{\mm}{\mathfrak{m}}

\def\cocoa{{\hbox{\rm C\kern-.13em o\kern-.07em C\kern-.13em o\kern-.15em A}}}

\newcommand{\calS}{\mathcal{S}}

\newcommand{\fkm}{\mathfrak{m}}

\def\ol{\overline}

\def\BG{BG}

\begin{document}

\title{The Non-Pure Dual Exchange Property in Low Dimensions}

\author{Reza Abdolmaleki}
\address{Department of Mathematics, Institute for Advanced Studies in Basic Sciences (IASBS), Zanjan 45137-66731, Iran}
\address{Department of Mathematics, Lahore University of Management Sciences, DHA, Lahore Cantt. 54792, Lahore, Pakistan}
\email{reza.abd110@gmail.com}
\email{reza.abdolmaleki@lums.edu.pk}

\author{Shinya Kumashiro}
\address{Department of Mathematics, Osaka Institute of Technology, 5-16-1 Omiya, asahi-ku, Osaka, 535-8585, Japan}
\email{shinya.kumashiro@oit.ac.jp}
\email{shinyakumashiro@gmail.com}

\begin{abstract} 
We investigate monomial ideals satisfying the non-pure dual exchange property, a notion introduced in connection with componentwise polymatroidal ideals. Our contributions are twofold. First, we show that in two variables, every integrally closed monomial ideal satisfies this property; as a consequence, we characterize polymatroidal ideals in two variables. Second, for strongly stable (Borel) ideals in three variables, we establish a practical criterion: it suffices to verify the defining condition only for the Borel generators, and this verification reduces to simple inequalities involving the degrees in the second and third variables.
\end{abstract}


\subjclass[2020]{Primary:13F20; Secondary:13A30, 13E15, 13C14.}
\keywords{monomial ideal, exchange property, integrally closed ideal, Borel ideal}
\thanks{}

\maketitle

\section{Introduction}

Exchange-type conditions on monomial generators have played a key role in uncovering the structural and homological behavior of monomial ideals. Notably, polymatroidal and componentwise polymatroidal ideals can be characterized through such conditions and enjoy powerful consequences such as linear quotients and linear resolutions; see, for instance, the survey \cite{HH} and the introduction of \cite{BQ}. 

Motivated by these developments, Bandari and Qureshi \cite{BQ} introduced the \emph{non-pure dual exchange property} (see Definition \ref{d12}) to facilitate the study of componentwise polymatroidal ideals. In particular, they showed that this property coincides with the usual exchange property for single-degree ideals (thus characterizing polymatroidal ideals in that case) and that componentwise polymatroidal ideals satisfy the non-pure dual exchange property. 
 
 In the case of two variables, they further showed that the non-pure dual exchange property, the non-pure exchange property, and componentwise polymatroidality are equivalent. More recently, Bandari and Qureshi \cite{BQ2} established that ideals with the non-pure dual exchange property have linear quotients, and are therefore componentwise linear.
 
Against this background, the goal of this paper is to investigate monomial ideals satisfying the non-pure dual exchange property from a structural perspective, beyond componentwise linearity. Our main contributions are as follows. Let $S = \KK[x_1, \ldots, x_n]$ denote the polynomial ring over a field $\KK$.

\smallskip
\noindent\textbf{(1) Two variables.}
For \(n=2\), we show that every integrally closed monomial ideal satisfies the non-pure dual exchange property (\cref{thm23}), whereas the converse does not hold (\cref{convnot}). As a consequence of \cref{thm23}, we obtain an explicit characterization of polymatroidal ideals in \(\mathbb{K}[x,y]\). In particular, we prove that an equigenerated monomial ideal in two variables is polymatroidal if and only if it is integrally closed (\cref{polymatroidal}). 

\smallskip
\noindent\textbf{(2) Three variables.}
The behavior observed in the two-variable case does not fully extend to
monomial ideals in three variables; see \cref{e29} and \cref{colex}. This
motivates our focus on an important subclass, namely Borel ideals in three
variables. Recall that a monomial ideal \(I\) is a \textit{Borel ideal} if it is
closed under Borel moves: that is, whenever \(u \in I\) is a monomial and
\(x_j \mid u\), one has
\[
\frac{x_i}{x_j}u \in I \quad \text{for all } i<j.
\]

A natural question is whether Borel ideals satisfy the non-pure dual exchange
property. This is indeed the case for \emph{principal Borel ideals}, that is,
Borel ideals with a single Borel generator; see \cref{principal}. However, the
property fails for Borel ideals in general, as shown by the counterexample in
\cref{e37}(b).

The main result in the three-variable case is \cref{t37}. It shows that, for a
Borel ideal \(I \subseteq \KK[x,y,z]\), the non-pure dual exchange property can
be tested solely on the Borel generators of \(I\). Moreover, this reduction
yields a practical criterion in terms of the \(x_2\)- and \(x_3\)-degrees of
those generators. We include examples illustrating how this criterion can be
applied.

Finally, this reduction is specific to the three-variable setting: in higher
dimensions, checking only the Borel generators is no longer sufficient; see
\cref{count-4var}. Thus \cref{t37} is best possible in this sense.

\smallskip
The paper is organized as follows. Section~\ref{sec:two} treats the two-variable case and proves that integrally closed monomial ideals satisfy the non-pure dual exchange property, as well as providing a characterization of polymatroidality. Section~\ref{3} addresses the three-variable case, establishes the reduction to Borel generators for Borel ideals, and derives a concrete degree criterion illustrated with examples.

Throughout, we make the following assumptions.

\begin{setup}
Let $S = \KK[x_1, \ldots, x_n]$ be the polynomial ring over a field $\KK$, equipped with the standard grading (i.e., $\deg(x_i) = 1$ for each $i$). Let $\mm := (x_1, \ldots, x_n)$ denote the graded maximal ideal of $S$. For a monomial ideal $I \subseteq S$, let $G(I)$ (resp. $M(I)$) denote the minimal set of monomial generators of $I$ (resp. the set of all monomials in $I$).

Let $u = x_1^{a_1} \cdots x_n^{a_n}$ be a monomial in $S$. We denote by $\deg(u)$ the total degree of $u$, that is, 
\[
\deg(u) = a_1 + \cdots + a_n.
\]
Similarly, we denote by $\deg_{x_i}(u)$ the exponent of $x_i$ in $u$, i.e., $\deg_{x_i}(u) = a_i$. The vector $(a_1, \ldots, a_n) \in \mathbb{N}^n$ is called the \emph{exponent vector} of $u$. For a monomial ideal $I \subseteq S$, the set of exponent vectors of all monomials in $I$ is called the \emph{exponent set} of $I$ and is denoted by $\mathrm{Exp}(I)$.

\end{setup}

\begin{Definition}(\cite[Definition 1.3]{BQ})\label{d12}
Let $I$ be a monomial ideal in $S$. The ideal $I$ is said to satisfy the \emph{non-pure dual exchange property} if, for all $u, v \in G(I)$ with $\deg(u) \le \deg(v)$ and for all $i$ such that $\deg_{x_i}(v) < \deg_{x_i}(u)$, there exists $j$ with $\deg_{x_j}(v) > \deg_{x_j}(u)$ such that $x_i (v / x_j) \in I$.
\end{Definition}

For simplicity, we say that monomials $u, v \in G(I)$ satisfy \emph{the defining condition of the non-pure dual exchange property} if $\deg(u) \le \deg(v)$ and for all $i$ such that $\deg_{x_i}(v) < \deg_{x_i}(u)$, there exists $j$ with $\deg_{x_j}(v) > \deg_{x_j}(u)$ such that $x_i (v / x_j) \in I$.

\section{The property in two variables}\label{sec:two}

In this section, we assume \( n = 2 \) and write \( x_1 = x \) and \( x_2 = y \). 

Let \( I \subseteq \mathbb{K}[x,y] \) be a monomial ideal. 
We show that if \( I \) is integrally closed, then it satisfies the non-pure dual exchange property (\cref{thm23}). 
However, the converse does not hold (see \cref{convnot}). 
As a consequence, we give an explicit characterization of polymatroidal ideals in \( \mathbb{K}[x,y] \). 
In particular, we show that an equigenerated monomial ideal in \( \mathbb{K}[x,y] \) is polymatroidal if and only if it is integrally closed (\cref{polymatroidal}).

To prepare for this purpose, we introduce some notation. Let 
\[
G(I) = \{u_1, u_2, \ldots, u_r\}.
\] 
Without loss of generality, we may assume that $u_1, \dots, u_r$ have the form
\begin{equation}
\label{1}
u_1 = x^{a_1} y^{b_1}, \; u_2 = x^{a_2} y^{b_2}, \; \dots, \; u_r = x^{a_r} y^{b_r},
\end{equation}
where $a_1 > a_2 > \cdots > a_r$ and $b_1 < b_2 < \cdots < b_r$. 

Let 
\[
\mathcal{S} = \{m_i\}_{i=1}^{r-1}
\] 
be the sequence of ordered pairs defined by $m_i = (a_i - a_{i+1},\; b_i - b_{i+1})$ for $i = 1, \ldots, r-1$. We start with the following proposition.

\begin{prop} 
\label{notproperty}
Let $I$ be a monomial ideal in $S=\KK[x,y]$. Then, the following statements are equivalent:
\begin{enumerate}[\rm(a)]
\item $I$ does not satisfy the non-pure dual exchange property.
\item  After ignoring all ordered pairs of the form $(1, -1)$ in $\mathcal{S}$, the sequence $\mathcal{S}$ contains either a pair of the form $(\alpha, -\beta)$ with $\alpha, \beta \geq 2$, or a subsequence $\mathcal{S}_1 = \{(1, -\beta), (\alpha, -1)\}$ with $\alpha, \beta \geq 2$.
   \end{enumerate}
\end{prop}
\begin{proof}
(b) $\implies$ (a):  Suppose that $\mathcal{S}$ contains a pair of the form $(\alpha, -\beta)$ with $\alpha, \beta \geq 2$. We write $(\alpha, -\beta) = m_\ell$ for some $\ell \in \{1, \ldots , r-1\}$. Then, 
\begin{center}
$u_\ell=x^{a_\ell} y^{b_\ell}=x^{a_{\ell+1}+\alpha} y^{b_{\ell+1}-\beta}$ \qquad and \qquad $u_{\ell+1}=x^{a_{\ell+1}} y^{b_{\ell+1}} = x^{a_\ell - \alpha} y^{b_\ell+\beta}$. 
\end{center}
We then show that $u_\ell, u_{\ell+1}\in G(I)$ or $u_{\ell+1}, u_\ell\in G(I)$ do not satisfy the defining condition of the non-pure dual exchange property. 

Indeed, if $\deg(u_\ell) \leq \deg(u_{\ell+1})$, then we observe that $\deg_{x}(u_{\ell+1}) < \deg_{x}(u_\ell)$ and $x(u_{\ell+1}/y) = x^{a_{\ell}-\alpha+1} y^{b_{\ell}+\beta-1} \notin I$ since $\alpha>1$. 
If $\deg(u_{\ell+1}) \leq \deg(u_{\ell})$, then $\deg_{y}(u_\ell) < \deg_{y}(u_{\ell+1})$ and $y(u_{\ell}/x) = x^{a_{\ell+1}+\alpha-1} y^{b_{\ell+1}-\beta+1}\notin I$ since $\beta\ge 2$. 


Therefore, $I$ does not satisfy the non-pure dual exchange property.

Now, suppose that $\mathcal{S}$ does not include any pair of the form $(\alpha, -\beta)$ with $\alpha, \beta \geq 2$, but after ignoring all ordered pairs of the form $(1, -1)$, contains a subsequence $\mathcal{S}_1 = \{(1, -\beta), (\alpha, -1)\}$ with $\alpha, \beta \geq 2$. We fix indices $i, j \in \{1, \ldots, r-1\}$ with $j > i$, such that
$m_i = (a_i - a_{i+1},\, b_i - b_{i+1}) = (1, -\beta)$ and $m_j = (a_j - a_{j+1},\, b_j - b_{j+1}) = (\alpha, -1)$. We show that $u_i, u_{j+1}\in G(I)$ or $u_{j+1}, u_i\in G(I)$ do not satisfy the defining condition of the non-pure dual exchange property. We note $u_i=x^{a_i} y^{b_i} =x^{a_{i+1}+1} y^{b_{i+1}-\beta}$ and $u_{j+1} = x^{a_{j+1}} y^{b_{j+1}} = x^{a_{j}-\alpha} y^{b_{j}+1}$.

If $\deg(u_i) \leq \deg(u_{j+1})$, then we get $\deg_{x}(u_{j+1}) < \deg_{x}(u_i)$ and $x(u_{j+1}/y) = x^{a_{j}+1-\alpha} y^{b_{j}}\notin I$ since $\alpha>1$. 
If $\deg(u_{j+1}) \leq \deg(u_{i})$, we get $\deg_{y}(u_i) < \deg_{y}(u_{j+1})$ and $y(u_i/x) = x^{a_{i+1}} y^{b_{i+1}+1-\beta}\notin I$ since  $\beta>1$. 

Therefore, $I$ does not satisfy the non-pure dual exchange property.

%

%

(a) $\implies$ (b): Assume that $I$ does not satisfy the non-pure dual exchange property. Then, there exist monomials $u, v \in G(I)$ with $\deg(u) \leq \deg(v)$ such that $u, v \in G(I)$ do not satisfy the defining condition of the non-pure dual exchange property. 

Assume that $\deg_{x}(u) > \deg_{x}(v)$, and fix indices $i, j \in \{1, \ldots, r\}$ with $j > i$ such that 
$u = x^{a_i} y^{b_i}$ and $v = x^{a_j} y^{b_j}$. Then
\begin{align}
\label{2}
x^{a_j+1} y^{b_j-1} \notin I.
\end{align}
It follows that $a_{j-1} - a_j \geq 2$. 
Indeed, if $a_{j-1} - a_j = 1$, then 
\[
x^{a_{j-1}} y^{b_{j-1}} = x^{a_j+1} y^{b_{j-1}}
\] 
divides $x^{a_j+1} y^{b_j-1}$ (since $b_{j-1} \le b_j - 1$), which contradicts \eqref{2}. 
Therefore, we must have $a_{j-1} - a_j \geq 2$.

If $b_j - b_{j-1} \geq 2$, then combining this with the inequality $a_{j-1} - a_j \geq 2$ yields 
\[
(a_{j-1} - a_j,\, b_{j-1} - b_j) = (\alpha, -\beta) \quad \text{with } \alpha, \beta \geq 2,
\] 
as desired.

Now, assume that $b_j - b_{j-1} = 1$ and that there exists no pair $(\alpha, -\beta)$ with $\alpha, \beta \geq 2$ in $\mathcal{S}$. 
Then we have 
\[
(a_{j-1} - a_j,\, b_{j-1} - b_j) = (\alpha, -1) \quad \text{with } \alpha \geq 2.
\] 
Moreover, all other ordered pairs $m_k \in \mathcal{S}$ with $i \le k < j$ are of one of the following forms: 
\[
(1, -1), \quad (\alpha, -1) \text{ with } \alpha \ge 2, \quad \text{or } (1, -\beta) \text{ with } \beta \ge 2.
\]

It is enough to show that there exists an integer $k$ with $i \leq k < j$ such that $m_k = (1, -\beta)$ for some integer $\beta \geq 2$. Suppose that for all $k$ with $i \leq k < j$, we have $m_k = (c_k, -1)$ with $c_k \geq 1$. Note that in particular $c_{j-1} = \alpha \geq 2$. Then, it follows that 
\begin{equation*}
(a_i,b_i)=(a_j,b_j)+\sum_{k=i}^{j-1}m_k=\left(a_j+\sum_{k=i}^{j-1}c_k, \, b_j-(j-i)\right).
\end{equation*}
Therefore, 
\begin{equation*}
\deg(u)=a_i+b_i=a_j+\sum_{k=i}^{j-1}c_k+ b_j-(j-i)=a_j+b_j+\sum_{k=i}^{j-1}c_k-(j-i)=\deg(v)+\sum_{k=i}^{j-1}c_k-(j-i).
\end{equation*}
Since $c_{j-1} \geq 2$, it follows that 
\[
\sum_{k=i}^{j-1} c_k - (j - i) \geq 1.
\]
Hence, 
\[
\deg(v) + \sum_{k=i}^{j-1} c_k - (j - i) > \deg(v),
\]
which is a contradiction. Therefore, there exists an integer $k$ with $i \leq k < j$ such that $m_k = (1, -\beta)$ for some integer $\beta \geq 2$, as desired. 

The case where $\deg_{x}(u) < \deg_{x}(v)$ can be proved similarly. 
\end{proof}

By considering the contraposition of \cref{notproperty}, we obtain the following result. We remark that \cref{okproperty} is essentially known from \cite[Theorem 2.6]{BQ}, although the notation there is different.

\begin{cor}{\rm (cf. \cite[Theorem 2.6]{BQ})}
\label{okproperty}
Let $I$ be a monomial ideal in $S=\KK[x,y]$. Then, the following statements are equivalent:
\begin{enumerate}[\rm(a)]
\item $I$ satisfies the non-pure dual exchange property.
 \item  The sequence $\mathcal{S}$ has the form
\[
\mathcal{S} = \left\{ (\alpha_1, -1), \ldots, (\alpha_s, -1), (1, -\alpha_{s+1}), \ldots, (1, -\alpha_{r-1}) \right\}
\]
with $\alpha_s \geq 1$ for all $s \in \{1, \ldots, r-1\}$.
   \end{enumerate}
\end{cor}


Let us recall the notion of integral closedness.

\begin{Definition}(e.g., \cite[Definition 1.1.1]{HS})
Let $I$ be an ideal in a Noetherian ring $R$. An element $z \in R$ is said to be 
\emph{integral over} $I$ if there exist an integer $s$ and elements $a_i \in I^i$, $i = 1, \ldots, s$, 
such that
\[
z^s + a_1 z^{s-1} + a_2 z^{s-2} + \cdots + a_{s-1} z + a_s = 0.
\]

The set of all elements that are integral over $I$ is called the \emph{integral closure} of $I$, and is denoted by $\overline{I}$. If $I = \overline{I}$, then $I$ is called \emph{integrally closed}.
\end{Definition}

\begin{prop}[{\cite[Proposition~1.4.6]{HS}}]
\label{intclose}
The exponent set of the integral closure of a monomial ideal $I$ equals all the integer lattice points in the convex hull of the exponent set of $I$.
\end{prop}

\begin{thm}\label{thm23}
Let $I \subset S = \mathbb{K}[x,y]$ be an integrally closed monomial ideal. 
Then $I$ satisfies the non-pure dual exchange property.
\end{thm}
\begin{proof}
 Assume that $I$ is integrally closed. If $I$ does not satisfy the non-pure dual exchange property, by \cref{notproperty}, there exist integers $\alpha, \beta\geq 2$ and a subsequence of $\calS$ such that: 
\begin{center}
 (i) $\{(\alpha, -\beta)\}$ \quad  or \quad  (ii) $\{(1, -\beta), (1, -1), \dots,(1, -1), (\alpha, -1)\}$. 
\end{center}

 Case (i): Let $m_i = (a_i - a_{i+1},\; b_i - b_{i+1})= (\alpha, -\beta)$ for some $i \in\{ 1, \ldots, r-1\}$. Then, we show that $x^{a_{i}-1}y^{b_{i+1}-1}\not\in I$. Indeed, if $x^{a_{i}-1}y^{b_{i+1}-1}\in I$, then there exists $u_j=x^{a_j}y^{b_j}\in G(I)$ such that $u_j | x^{a_{i}-1}y^{b_{i+1}-1}$, that is, $a_j\le a_{i}-1<a_i$ and $b_j\le b_{i+1}-1<b_{i+1}$, which is impossible (see \eqref{1}). Thus, $x^{a_{i}-1}y^{b_{i+1}-1}\not\in I$. On the other hand, since we have 
\begin{center}
$u_i=x^{a_i}y^{b_i} =x^{a_{i}}y^{b_{i+1}-\beta}\in G(I)$, \quad $u_{i+1}=x^{a_{i+1}}y^{b_{i+1}} =x^{a_{i}-\alpha}y^{b_{i+1}}\in G(I)$, \quad and $\alpha, \beta\ge 2$, 
\end{center}
by \cref{intclose}, it follows that
\[
x^{a_{i}-1}y^{b_{i+1}-1}\in \ol{(u_i, u_{i+1})} \subseteq \ol{I} = I,
\] 
a contradiction. 

 Case (ii): Choose $1\le i<j\le r-1$ such that $m_i=(1, -\beta)$ and $m_{j}=(\alpha, -1)$. 
We note that $x^{a_{i}-1}y^{b_{i}+\beta-1} = x^{a_{i+1}}y^{b_{i+1}-1}=(u_{i+1}/y)\notin I$. On the other hand, since $\alpha, \beta\ge 2$, one can easily check that 
\[
x^{a_{i}-1}y^{b_{i}+\beta-1}\in \ol{(x^{a_i}y^{b_i}, x^{a_j}y^{b_j})} \subseteq \ol{I} =I.
\]
This is a contradiction. 

Hence, in both cases, we obtain a contradiction. 
It follows that $\calS$ has the desired form, and therefore $I$ satisfies the non-pure dual exchange property.

\end{proof}

\begin{ex}
The ideal 
\[
I = (x^7,\, x^5y,\, x^4y^2,\, x^2y^3,\, xy^6,\, y^8)
\]
satisfies the non-pure dual exchange property. 

The exponent set of the minimal generators of \(I\) is
\[
\{ 
v_1=(7,0),\, v_2=(5,1),\, v_3=(4,2),\, v_4=(2,3),\, v_5=(1,6),\, v_6=(0,8)\}.
\]
Figure~\ref{fig:int} illustrates the lattice points in the convex hull of the exponent vectors of $I$ (equivalently, the exponent set of the integral closure of $I$). This figure visually confirms the property described in Theorem~\ref{thm23}.
\end{ex}

\begin{figure}[h!]
\centering
\begin{tikzpicture}[scale=0.6,>=stealth] 
  
   \draw[->] (-0.5,0) -- (8.5,0) node[right] {$x$-exponent};
  \draw[->] (0,-0.5) -- (0,9.5) node[above] {$y$-exponent};

  \foreach \x in {0,1,...,8}
    \draw (\x,0) -- (\x,-0.15) node[below,font=\small] {\x};
  \foreach \y in {0,1,...,9}
    \draw (0,\y) -- (-0.15,\y) node[left,font=\small] {\y};

  \foreach \x/\y/\name in {7/0/v_1, 5/1/v_2, 4/2/v_3, 2/3/v_4, 1/6/v_5, 0/8/v_6}{
    \filldraw[red] (\x,\y) circle (2pt);
    \node[anchor=west,font=\small] at (\x+0.2,\y) {$\name$};
  }

   \draw[thick,blue]
    (7,0) -- (5,1) -- (4,2) -- (2,3) -- (1,6) -- (0,8);

  \fill[blue!20,opacity=0.5]
    (7,0) -- (5,1) -- (4,2) -- (2,3) -- (1,6) -- (0,8)
    -- (0,9.5) -- (8.5,9.5) -- (8.5,0) -- cycle;

  \draw[step=1cm,gray,very thin] (0,0) grid (8,9);

\end{tikzpicture}

\vspace{2mm}
\caption{The exponent set of the integral closure of $I$.}
\label{fig:int}
\end{figure}
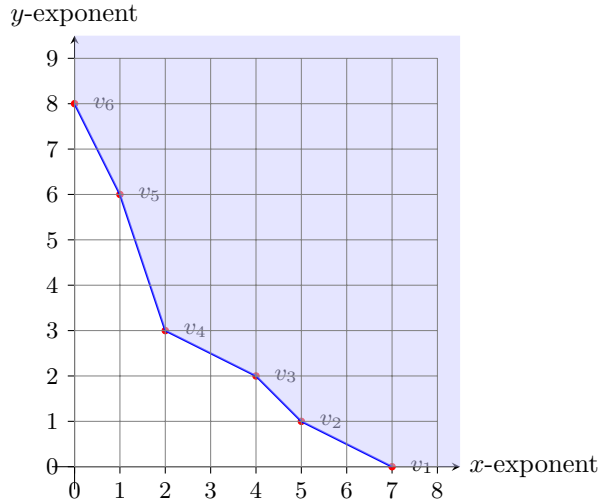

\begin{rem}
\label{convnot}
The converse of \cref{thm23} does not hold in general. 
For instance, consider the ideal
\[
I = (x^7, x^6y, x^5y^2, x^4y^3, y^4) \subset \KK[x,y].
\]
It is easy to check that $I$ satisfies the non-pure dual exchange property, 
but it is not integrally closed, since 
\[
x^3y^3 \in \overline{I} \setminus I.
\]
\end{rem}

The following example shows that the assertion of Theorem~\ref{thm23} may fail when \( n \ge 3 \):

\begin{ex}\label{e29}
 Consider the ideal \( I = (x^2, xy, xz, y^2) \subset \KK[x, y, z] \).  
Set \( u = xz \) and \( v = y^2 \). Since $\deg_z (u) > \deg_z(v)$ and  $z(v/y) = yz \notin I$, 
the ideal \( I \) does not satisfy the non-pure dual exchange property.  
However, it is easy to verify (either manually or using CoCoA~\cite{Co}) that \( I \) is integrally closed.
\end{ex}

We recall from \cite[Proposition~1.4]{BQ} that an equigenerated monomial ideal is \textit{polymatroidal} if and only if it satisfies the non-pure dual exchange property. Polymatroidal ideals have been well studied; see, for example, \cite[Chapter 12]{HH}. Theorem~\ref{thm23} yields the following consequence, which characterizes polymatroidal ideals in $\KK[x,y]$.

\begin{thm}
\label{polymatroidal}
Let $I \subset \mathbb{K}[x,y]$ be an equigenerated monomial ideal in degree $d$. The following are equivalent:
\begin{enumerate}[\rm(a)]
\item $I$ is integrally closed.
\item $I$ is polymatroidal.
\item 
\[
I=(x^iy^{d-i} \mid d_1\le i \le d_2)
\]
for some integers $d_1, d_2$ with $0\le d_1 \le d_2\le d$.
\end{enumerate}
\end{thm}
\begin{proof}
(a) $\implies$  (b)
Let $I$ be an equigenerated integrally closed monomial ideal. By Theorem~\ref{thm23}, $I$ satisfies the non-pure dual exchange property and hence, by \cite[Proposition~1.4]{BQ}, is polymatroidal.

(b)$\implies$ (c) By \cref{okproperty}, the sequence $\mathcal{S}$ has the form
\[
\mathcal{S} = \left\{ (\alpha_1, -1), \ldots, (\alpha_s, -1), (1, -\alpha_{s+1}), \ldots, (1, -\alpha_{r-1}) \right\}
\]
with $\alpha_i \geq 1$ for all $i \in \{1, \ldots, r-1\}$. Since $I$ is equigenerated, $\alpha_i=1$ for all $i$. This implies the assertion.

%

(c) $\implies$ (a) Let
\[
I=(x^iy^{d-i} \mid d_1\le i \le d_2).
\]
Then the exponent set of minimal generators of $I$ coincides with all integer lattice points on a line segment. Hence, by \cref{intclose}, $I$ is integrally closed.
\end{proof}

\begin{rem}
A classical theorem of Zariski asserts that, in a two-dimensional regular local ring, the product of two integrally closed ideals is again integrally closed (see, e.g., \cite[Theorem 14.4.4]{HS}). Hence, by \cref{polymatroidal}, the class of polymatroidal ideals in $\KK[x,y]$ is closed under products. This closure property is also a special case of \cite[Theorem~12.6.3]{HH}.
\end{rem}

%

We also apply \cref{thm23} to the notion of saturation. 
Recall that in a Noetherian graded ring $R$ having the unique graded maximal ideal $\mm$, the ideal 
\[
I:\mm^{\infty} = \bigcup_{k \geq 0} (I : \fkm^k)
\] 
is called the \textit{saturation} of $I$ and is denoted by $I^{\text{sat}}$ (\cite[Section 1.2.2]{HH}).


\begin{rem}\label{rem211}
If $I$ is an integrally closed ideal, then so is $I:\fkm$. Indeed, by {\cite[Remark~1.3.2 (2)]{HS}}, we have
\[
\overline{I : \mm} \subseteq \overline{I} : \mm = I : \mm.
\]
The reverse inclusion \( I : \mm \subseteq \overline{I : \mm} \) always holds. 
\end{rem}

\begin{cor}\label{colon}
$I:\fkm$ and $I^{\text{sat}}$ satisfy the non-pure dual exchange property for all integrally closed monomial ideals $I$ in $\KK[x,y]$. 
\end{cor}

\begin{proof}
$I:\fkm$ satisfies the non-pure dual exchange property by \cref{rem211} and \cref{thm23}. Since $I:\mm^k = (I:\mm^{k-1}):\mm$, using induction, it follows that \( I^{\text{sat}} \) also satisfies the non-pure dual exchange property.
\end{proof}

The statement of \cref{colon} does not hold in polynomial rings with more than two variables:

\begin{ex}
\label{colex}
The ideal $I = (x^3, x^2y, xyz, x^2z^5, y^3z^3) \subset \KK[x,y,z]$ is integrally closed. One can check (for instance, using CoCoA (\cite{Co})) that 
\[
G(I:\mm) = \{x^3, x^2y, xyz, x^2z^4, y^3z^3\}.
\] 
Consider $u = x^2z^4$ and $v = y^3z^3$. Then $\deg(u) = \deg(v)$ and $\deg_{z}(u) > \deg_{z}(v)$, but $z(v/y) = y^2z^4 \notin I:\mm$. Therefore, $I:\mm$  does not satisfy the non-pure dual exchange property.
\end{ex}

\section{The property in three variables}
\label{3}
In this section, we study the non-pure dual exchange property in dimension three.
As we saw in the previous section, not every result from the two-variable case extends to monomial ideals in three variables (see \cref{e29} and \cref{colex}).
Therefore, we focus on an important class of monomial ideals, namely Borel ideals in three variables.
The main objective of this section is to establish \cref{t37}, which provides a simpler criterion for the non-pure dual exchange property. More precisely, we prove that for a Borel ideal \(I\) in three variables, it is enough to check the defining condition of the property only on the Borel generators of \(I\).


\begin{defn}
An ideal \( I \subseteq S = \KK[x_1, \ldots, x_n] \) is called \emph{strongly stable} (or a \emph{Borel ideal}) if \( I \) is closed under Borel moves. 
That is, if \( u \in M(I) \) and \( x_j \mid u \), then \( \frac{x_i}{x_j}u \in I \) for all \( i < j \). 
For a Borel ideal \( I \), the \emph{Borel generators} of \( I \), denoted by \( \BG(I) \), form the unique minimal set of monomials whose Borel moves generate \( I \). If \( u_1, \ldots, u_m \) are the Borel generators of \( I \), we write \( I = B(u_1, \ldots, u_m) \). A Borel ideal with a single Borel generator is called a \emph{principal Borel ideal}. 
\end{defn}

In this section, let \( n = 3 \), and we set \( x_1 = x \), \( x_2 = y \), and \( x_3 = z \); that is, \( S = \KK[x, y, z] \). We divide the conditions of the non-pure dual exchange property as follows.

\begin{defn}\label{1.2}
Let $u,v$ be monomials with $\deg(u) \le \deg(v)$. We define the following conditions on $u$ and $v$:
\begin{enumerate}
\item[\rm(a)] If $\deg_x(u) > \deg_x(v)$, then at least one of the following assertions holds:
\begin{enumerate}
\item[\rm(a-1)]  $\deg_y(u) < \deg_y(v)$ and $\dfrac{x}{y}v\in I$, or 
\item[\rm(a-2)]  $\deg_z(u) < \deg_z(v)$ and $\dfrac{x}{z}v\in I$.
\end{enumerate}
\item[\rm(b)]  If $\deg_y(u) > \deg_y(v)$, then at least one of the following assertions holds:
\begin{enumerate}
\item[\rm(b-1)]  $\deg_x(u) < \deg_x(v)$ and $\dfrac{y}{x}v\in I$, or 
\item[\rm(b-2)]  $\deg_z(u) < \deg_z(v)$ and $\dfrac{y}{z}v\in I$.
\end{enumerate}
\item[\rm(c)]  If $\deg_z(u) > \deg_z(v)$, then at least one of the following assertions holds:
\begin{enumerate}
\item[\rm(c-1)]  $\deg_x(u) < \deg_x(v)$ and $\dfrac{z}{x}v\in I$, or 
\item[\rm(c-2)]  $\deg_y(u) < \deg_y(v)$ and $\dfrac{z}{y}v\in I$.
\end{enumerate}
\end{enumerate}

\end{defn}

Of course, \( I \) satisfies the non-pure dual exchange property if and only if, for all \( u, v \in G(I) \) with \( \deg(u) \le \deg(v) \), the monomials \( u \) and \( v \) satisfy the conditions (a), (b), and (c).

\begin{rem}\label{r1.3}
Suppose that $ I \subseteq S $ is a Borel ideal. Then, $I$ satisfies the non-pure dual exchange property  if and only if, for all $u,v\in G(I)$ with $\deg(u)\le \deg(v)$, the monomials $u$ and $v$ satisfy the conditions (b) and (c).
\end{rem}

\begin{proof}
We show that if \( I \) is a Borel ideal, then condition (a) is automatically satisfied for all \( u, v \in G(I) \) with \( \deg(u) \le \deg(v) \). 
Assume that \( \deg_x(u) > \deg_x(v) \). Since \( \deg(u) \le \deg(v) \), at least one of the inequalities \( \deg_y(u) < \deg_y(v) \) or \( \deg_z(u) < \deg_z(v) \) must hold. 
If \( \deg_y(u) < \deg_y(v) \), then \( \deg_y(v) > 0 \); hence, since \( I \) is a Borel ideal, we get \( \frac{x}{y}v \in I \). 
Similarly, if \( \deg_z(u) < \deg_z(v) \), then \( \frac{x}{z}v \in I \).
\end{proof}

\begin{prop}\label{p1.4}
Let $ I \subseteq S $ be a Borel ideal. The following conditions are equivalent:

\begin{enumerate}[\rm(1)] 
\item $I$ satisfies the non-pure dual exchange property.
\item For all $u\in G(I)$ and $v\in \BG(I)$ with $\deg(u)\le \deg(v)$, the monomials $u$ and $v$ satisfy the conditions (b) and (c) of \cref{1.2}.
\end{enumerate}
\end{prop}

\begin{proof}
(1) $\Longrightarrow$ (2) is clear by Remark~\ref{r1.3}.

(2) $\Longrightarrow$ (1): We prove that for all $u\in G(I)$ and $v\in G(I)\setminus \BG(I)$ with $\deg(u)\le \deg(v)$, the monomials $u$ and $v$ satisfy the conditions (b) and (c). Write $u=x^{a_1}y^{a_2}z^{a_3}$ and $v=x^{b_1}y^{b_2}z^{b_3}$ for some non-negative integers $a_1, a_2, a_3, b_1, b_2, b_3$. 

Condition (b): Assume that $\deg_y(u) > \deg_y(v)$. Since $\deg(u)\le \deg(v)$, at least one of the conditions $\deg_x(u)< \deg_x(v)$ or $\deg_z(u)< \deg_z(v)$ holds. If $\deg_z(u)< \deg_z(v)$, then $\dfrac{y}{z}v\in I$ since $I$ is a Borel ideal. Thus, we may assume that $\deg_x(u)< \deg_x(v)$ and $\deg_z(u)\ge  \deg_z(v)$. That is, we have $a_1+a_2+a_3\le b_1+b_2+b_3$, $a_1<b_1$, $a_2>b_2$, and $a_3\ge b_3$. Hence, we get 
\[
\dfrac{y}{x} v = x^{(b_1+b_2+b_3)-(a_1+a_2+a_3)} \left(\dfrac{x}{y}\right)^{((a_2-b_2-1)+(a_3-b_3))}\left(\dfrac{y}{z}\right)^{a_3-b_3}u\in I. 
\]

Condition (c): Assume that $\deg_z(u) > \deg_z(v)$. Since $\deg(u)\le \deg(v)$, at least one of the conditions $\deg_x(u)< \deg_x(v)$ or $\deg_y(u)< \deg_y(v)$ holds. Assume that $\deg_x(u)< \deg_x(v)$. Thus, we have $a_1+a_2+a_3\le b_1+b_2+b_3$, $a_1<b_1$, and $a_3>b_3$. Hence, we get 
\[
\dfrac{z}{x} v = y^{(b_1+b_2+b_3)-(a_1+a_2+a_3)} \left(\dfrac{x}{y}\right)^{b_1-a_1-1}\left(\dfrac{y}{z}\right)^{a_3-b_3-1}u\in I. 
\]
Therefore, we may assume that $\deg_x(u)\ge \deg_x(v)$. Then, $\deg_y(u)< \deg_y(v)$. Thus, we have $a_1\ge b_1$, $a_2<b_2$, and $a_3>b_3$. Note that there exists $w\in \BG(I)$ such that $v$ is obtained from Borel moves on $w$, that is, there exist non-negative integers $\alpha_1, \alpha_2, \alpha_3$ such that
\[
v=\left(\dfrac{x}{y}\right)^{\alpha_1}\left(\dfrac{x}{z}\right)^{\alpha_2}\left(\dfrac{y}{z}\right)^{\alpha_3}w.
\]
If $\alpha_2>0$, then $\dfrac{z}{y}v = \left(\dfrac{x}{y}\right)^{\alpha_1+1}\left(\dfrac{x}{z}\right)^{\alpha_2-1}\left(\dfrac{y}{z}\right)^{\alpha_3}w\in I$. If $\alpha_3>0$, then $\dfrac{z}{y}v = \left(\dfrac{x}{y}\right)^{\alpha_1}\left(\dfrac{x}{z}\right)^{\alpha_2}\left(\dfrac{y}{z}\right)^{\alpha_3-1}w\in I$. Hence, we may assume that $\alpha_2=\alpha_3=0$. We have $\alpha_1>0$ since $v\not\in \BG(I)$.
Then, $w=x^{b_1-\alpha_1}y^{b_2+\alpha_1}z^{b_3}$. 
It follows that $\deg(u) \le \deg(v)=\deg(w)$, $\deg_x(u)\ge \deg_x(v)>\deg_x(w)$, $\deg_y(u)<\deg_y(v)<\deg_y(w)$, and $\deg_z(u)>\deg_z(v)=\deg_z(w)$. Since we assume (2) and \( w \in \BG(I) \), we have \( \frac{z}{y}w \in I \). 
As \( I \) is closed under Borel moves, we obtain 
\[
\frac{z}{y}v = \left(\frac{x}{y}\right)^{\alpha_1}\frac{z}{y}w \in I,
\]
as desired.
\end{proof}

\begin{rem}\label{r1.5} 
Let $I$ be a Borel ideal. Suppose that $u \in G(I)$ and $v \in \BG(I)$ with $\deg(u) \le \deg(v)$. Then the following statements hold.
\begin{enumerate}[\rm(1)] 
\item The monomials $u$ and $v$ satisfy condition (b) of \cref{1.2} if and only if the following holds:

(b$^{\prime}$) \quad  If  $\deg_y(u) > \deg_y(v)$, then $\deg_z(u) < \deg_z(v)$.

\item The monomials $u$ and $v$ satisfy condition (c) of \cref{1.2} if and only if the following holds:

(c$^{\prime}$) \quad $\deg_z(u) \le \deg_z(v)$.
\end{enumerate}
\end{rem}

\begin{proof}
(b) $\Longleftrightarrow$ (b$^{\prime}$): 
If $\dfrac{y}{x}v \in I$, then $v = \dfrac{x}{y} \left(\dfrac{y}{x}v\right)$, 
which contradicts the assumption that $v \in \BG(I)$. 
Hence, the case (b-1) does not occur. 
Furthermore, if $\deg_z(u) < \deg_z(v)$, then $\deg_z(v) > 0$, 
and since $I$ is closed under Borel moves, it follows that $\dfrac{y}{z}v \in I$. 
Therefore, the assertion holds.

(c) $\Longleftrightarrow$ (c$^{\prime}$): 
For the same reason as above, by the assumption that $v \in \BG(I)$, 
we have $\dfrac{z}{x}v \notin I$ and $\dfrac{z}{y}v \notin I$. 
Hence, condition (c) is equivalent to saying that $\deg_z(u) > \deg_z(v)$ does not hold.
\end{proof}

\begin{prop}\label{p1.6}
Let $I$ be a Borel ideal. The following are equivalent: 
\begin{enumerate}[\rm(1)] 
\item For all $u\in G(I)$ and $v\in \BG(I)$ with $\deg(u)\le \deg(v)$, the monomials $u$ and $v$ satisfy the conditions (b$^{\prime}$) and (c$^{\prime}$) of Definition \ref{1.2}.
\item For all $u, v\in \BG(I)$ with $\deg(u)\le \deg(v)$, the monomials $u$ and $v$ satisfy the conditions (b$^{\prime}$) and (c$^{\prime}$).
\end{enumerate}
\end{prop}

\begin{proof}
(1) $\Longrightarrow$ (2) is clear. 

(2) $\Longrightarrow$ (1): Write 
\[
u = \left(\frac{x}{y}\right)^{\alpha_1}\left(\frac{x}{z}\right)^{\alpha_2}\left(\frac{y}{z}\right)^{\alpha_3} w
\]
for some $w \in \BG(I)$. If $\alpha_2 > 0$, noting that $\frac{x}{y}\frac{y}{z} = \frac{x}{z}$, we have 
\[
u = \left(\frac{x}{y}\right)^{\alpha_1+\alpha_2}\left(\frac{y}{z}\right)^{\alpha_3+\alpha_2} w.
\]
Hence, by replacing $\alpha_1, \alpha_2, \alpha_3$ if necessary, we may assume that $\alpha_2 = 0$. 

We prove the assertion by induction on $\alpha := \alpha_1 + \alpha_3$.  
If $\alpha = 0$, then $u = w \in \BG(I)$, and thus the assertion holds by the hypothesis (2).  
Assume that $\alpha > 0$ and that the assertion holds for $\alpha - 1$.  
Since $\alpha > 0$, we have either $\alpha_1 > 0$ or $\alpha_3 > 0$. 

First, assume that $\alpha_1 > 0$. Then 
\[
\frac{y}{x}u = \left(\frac{x}{y}\right)^{\alpha_1-1}\left(\frac{y}{z}\right)^{\alpha_3} w \in I.
\]
By the induction hypothesis, $\frac{y}{x}u$ and $v$ satisfy (b$^{\prime}$) and (c$^{\prime}$); that is, 
\[
\deg_z\!\left(\frac{y}{x}u\right) \le \deg_z(v)
\quad \text{and} \quad
\deg_y\!\left(\frac{y}{x}u\right) > \deg_y(v)
\Rightarrow
\deg_z\!\left(\frac{y}{x}u\right) < \deg_z(v).
\]
Hence, $\deg_z(u) = \deg_z\!\left(\frac{y}{x}u\right) \le \deg_z(v)$. Moreover,
\[
\deg_y(u) > \deg_y(v)
\Longrightarrow
\deg_y\!\left(\frac{y}{x}u\right) > \deg_y(v)
\Longrightarrow
\deg_z(u) = \deg_z\!\left(\frac{y}{x}u\right) < \deg_z(v).
\]
Thus, $u$ and $v$ also satisfy (b$^{\prime}$) and (c$^{\prime}$). 

Next, assume that $\alpha_3 > 0$. Then $\frac{z}{y}u \in I$, and by the induction hypothesis, $\frac{z}{y}u$ and $v$ satisfy (b$^{\prime}$) and (c$^{\prime}$); that is,
\[
\deg_z\!\left(\frac{z}{y}u\right) \le \deg_z(v)
\quad \text{and} \quad
\deg_y\!\left(\frac{z}{y}u\right) > \deg_y(v)
\Rightarrow
\deg_z\!\left(\frac{z}{y}u\right) < \deg_z(v).
\]
It follows that $\deg_z(u) < \deg_z\!\left(\frac{z}{y}u\right) \le \deg_z(v)$.  
Hence, $u$ and $v$ also satisfy (b$^{\prime}$) and (c$^{\prime}$), as desired.
\end{proof}

Now, we arrive at the main result of this section, which reduces the verification of the non-pure dual exchange property to the Borel generators.

\begin{thm}\label{t37}
Let $I \subseteq S$ be a Borel ideal. Then the following conditions are equivalent.
\begin{enumerate}[\rm(a)]
\item \( I \) satisfies the non-pure dual exchange property.
\item For all \( u, v \in \BG(I) \) with \( \deg(u) \le \deg(v) \), the monomials \( u \) and \( v \) satisfy \( \deg_z(u) \le \deg_z(v) \). Moreover, if \( \deg_y(u) > \deg_y(v) \), then \( \deg_z(u) < \deg_z(v) \).
\end{enumerate}
\end{thm}

\begin{proof}
This follows from Propositions~\ref{p1.4} and~\ref{p1.6} and Remark~\ref{r1.5}.
\end{proof}
 
 \begin{rem}
 \label{count-4var} 
The conclusion of \cref{t37} does not extend to Borel ideals in more than three variables. Consider the polynomial ring $S=\KK[x_1,x_2,x_3,x_4]$
and the Borel ideal $I=B(x_1x_3^2, x_2^3x_4)$.
The Borel generators $x_1x_3^2$ and $x_2^3x_4$ satisfy the non-pure dual exchange property. However, $I$ itself does not satisfy this property.

Indeed, the minimal generating set of $I$ is
\[
\{x_1^3,x_1^2x_2,x_1^2x_3, x_1x_2^2, x_1x_2x_3, x_1x_3^2, x_2^4, x_2^3x_3, x_2^3x_4\}.
\]
Now consider the minimal generators $u=x_1x_3^2$ and $v=x_2^3x_3$. We have $\deg(u) = \deg(v)$ and $\deg_{x_3}(u) > \deg_{x_3}(v)$. However, the monomial $x_3(v/x_2)=x_2^2x_3^2 \notin I$. Therefore, $I$ does not satisfy the non-pure dual exchange property.
\end{rem}
 
It follows from \cref{t37} that:
\begin{cor}\label{principal}
Let \( I =B(u) \subseteq \KK[x, y, z] \) be a principal Borel ideal. Then \( I \) satisfies the non-pure dual exchange property.
\end{cor}
\begin{rem}\label{count-ex} 
Note that not every Borel ideal in $\KK[x, y, z]$ satisfying the non-pure dual exchange property is a principal Borel ideal. Moreover, not every Borel ideal in $\KK[x, y, z]$ satisfies the non-pure dual exchange property. The following examples illustrate these cases.
\end{rem}
\begin{ex}\label{e37}
\begin{enumerate}[\rm(a)]
\item It is easy to check that the Borel ideal \( I = B(x, y^2z) \subset \KK[x, y, z] \) satisfies the non-pure dual exchange property.

\item The Borel ideal \( I = (x^2, xy, xz, y^2)=B(xz,y^2) \subset \KK[x, y, z] \) does not satisfy the non-pure dual exchange property, as discussed in \cref{e29}.
\end{enumerate}
\end{ex}

\begin{rem}\label{tv} 
In contrast to \cref{e37}\,(b), which concerns Borel ideals in three variables, it is worth mentioning that Borel ideals in two variables always satisfy the non-pure dual exchange property. Indeed, in \( \KK[x, y] \), Borel ideals coincide with lexsegment ideals. By \cite[Lemma~4.3]{HMZ}, we have 
\[
I = (x^d,\, x^{d-1}y^{b_1},\, \ldots,\, x^{d-s}y^{b_s}) \quad \text{for integers } d > 0 \text{ and } 0 < b_1 < \cdots < b_s.
\]
Hence, the sequence \( \mathcal{S} = \{m_i\}_{i=1}^{r-1} \) takes the form 
\[
\mathcal{S} = \left\{ (1, -\alpha_1), \ldots, (1, -\alpha_{r-1}) \right\},
\]
where \( \alpha_i \geq 1 \) for all \( i \in \{1, \ldots, r-1\} \). It then follows from Corollary~\ref{okproperty} that \( I \) satisfies the non-pure dual exchange property.
\end{rem}

\begin{Acknowledgments}
The first author was supported by the Iran National Science Foundation (INSF) and the Institute for Advanced Studies in Basic Sciences (IASBS), Zanjan, Iran, under project number 4014109. This work was also partially supported by a grant from the IMU-CDC and the Simons Foundation. The first author would like to thank the second author for his kind hospitality during his visit to Osaka Institute of Technology. The second author was supported by JSPS KAKENHI Grant Number 24K16909.
\end{Acknowledgments}


\end{document}